\documentclass[12pt]{article}

\usepackage{amssymb}
\usepackage{amsthm}
\usepackage{amsmath}
\usepackage{graphicx}
\usepackage{fullpage}
\usepackage{color}
\usepackage{enumerate}
 \numberwithin{equation}{section}
\usepackage{booktabs}
\allowdisplaybreaks

 \usepackage{mathtools}
 \mathtoolsset{showonlyrefs}
\usepackage{placeins}

\theoremstyle{plain}
\newtheorem{thm}{Theorem}[section]
\newtheorem{cor}[thm]{Corollary}

\newtheorem{lem}[thm]{Lemma}
\newtheorem{prop}[thm]{Proposition}

\theoremstyle{definition}

\theoremstyle{remark}

\newtheorem{rem}[thm]{Remark}

\newcommand{\N}{\mathbb{N}}
\newcommand{\R}{\mathbb{R}}

\newcommand{\bp}{\begin{proof}[\ensuremath{\mathbf{Proof}}]}
\newcommand{\bs}{\begin{proof}[\ensuremath{\mathbf{Solution}}]}
\newcommand{\ep}{\end{proof}}
\newcommand{\be}{\begin{equation}}
\newcommand{\ee}{\end{equation}}

\begin{document}

% Title
\title{Asymptotic flatness of Morrey extremals}

% Name
\author{Ryan Hynd\footnote{Department of Mathematics, University of Pennsylvania.  Partially supported by NSF grant DMS-1554130.}\;  and Francis Seuffert\footnote{Department of Mathematics, University of Pennsylvania.}}

\maketitle 
 %  Abstract  
\begin{abstract} 
We study the limiting behavior as $|x|\rightarrow \infty$ of extremal functions $u$ for Morrey's inequality on $\R^n$.  In particular, we compute the limit of $u(x)$ 
as $|x|\rightarrow \infty$ and show $|x||Du(x)|$ tends to $0$. To this end, we exploit the fact that extremals are uniformly bounded and that they each satisfy a PDE of the form $-\Delta_pu=c(\delta_{x_0}-\delta_{y_0})$ for some $c\in \R$ and distinct $x_0,y_0\in \R^n$.  More generally,  we explain how to quantitatively deduce the asymptotic flatness of bounded $p$-harmonic functions on exterior domains of $\R^n$ for $p>n$.  
\end{abstract}

%%%%%%%%%%%%%%%%%%%%%%%%%%%%%%%%%%%%%
\section{Introduction}
% Morrey's inequality
For each $n\in \N$ and $p>n$, Morrey's inequality asserts that there is a constant $C>0$ such that
\be\label{MorreyIneq}
\sup_{x\ne y}\left\{\frac{|u(x)-u(y)|}{|x-y|^{1-n/p}}\right\}
\le C \left(\int_{\R^n}|Du|^pdx\right)^{1/p}
\ee
for all continuously differentiable functions $u:\R^n\rightarrow \R$. In particular, it provides control on the $1-n/p$ H\"older seminorm of any function whose first partial derivatives belong to $L^p(\R^n)$.  In recent work \cite{HyndSeuf}, we showed that there is a smallest constant $C_*>0$ for which \eqref{MorreyIneq} holds and that there are nonconstant functions for which equality holds in \eqref{MorreyIneq} with $C=C_*$.
We call any such function an {\it extremal}.

% PDE
\par It turns out that for any nonconstant extremal function $u$, there is a unique pair of distinct points $x_0, y_0\in \R^n$ such that 
\be\label{HolderMax}
\sup_{x\ne y}\left\{\frac{|u(x)-u(y)|}{|x-y|^{1-n/p}}\right\}=\frac{|u(x_0)-u(y_0)|}{|x_0-y_0|^{1-n/p}}.
\ee
Moreover, $u$ satisfies the PDE
\be\label{DipolePDE}
-\Delta_pu=c(\delta_{x_0}-\delta_{y_0})
\ee
in $\R^n$ for some nonzero constant $c$.  Here 
\be
\Delta_pv:=\text{div}(|Dv|^{p-2}Dv)
\ee
is the $p$-Laplacian, and equation \eqref{DipolePDE} is understood to mean
\be
\int_{\R^n}|Du|^{p-2}Du\cdot D\phi dx=c(\phi(x_0)-\phi(y_0))
\ee 
for each $\phi\in C^\infty_c(\R^n)$. 

% main theorem
\par Equation \eqref{DipolePDE} can be used to show that each extremal is bounded and has various symmetry properties. In this note, we will make use of these facts to prove the following theorem. We interpret the existence of limit \eqref{FlatMorrey} below as asserting that extremals are {\it asymptotically flat}. This result was also confirmed by numerical computations as observed in Figure \ref{DipoleFig}.  

\begin{thm}\label{mainThm}
Suppose $n\ge 2$ and that $p>n$. If $u$ is an extremal which satisfies \eqref{HolderMax}, then 
\be\label{FlatMorrey}
\lim_{|x|\rightarrow \infty}u(x)=\frac{1}{2}(u(x_0)+u(y_0))
\ee
and 
\be\label{DecayMorrey}
\lim_{|x|\rightarrow \infty}|x||Du(x)|=0.
\ee
Furthermore, 
\be
r^{p-n}\int_{|x|>r}|Du|^pdx=p\int_{|x|>r}|x|^{p-n}|Du|^{p-2}\left(Du\cdot \frac{x}{|x|}\right)^2dx
\ee
is nonincreasing in $r\in (s,\infty)$ for some $s>0$ and tends to $0$ as $r\rightarrow\infty$. 
\end{thm}

\begin{figure}[h]
\centering
 \includegraphics[width=1\columnwidth]{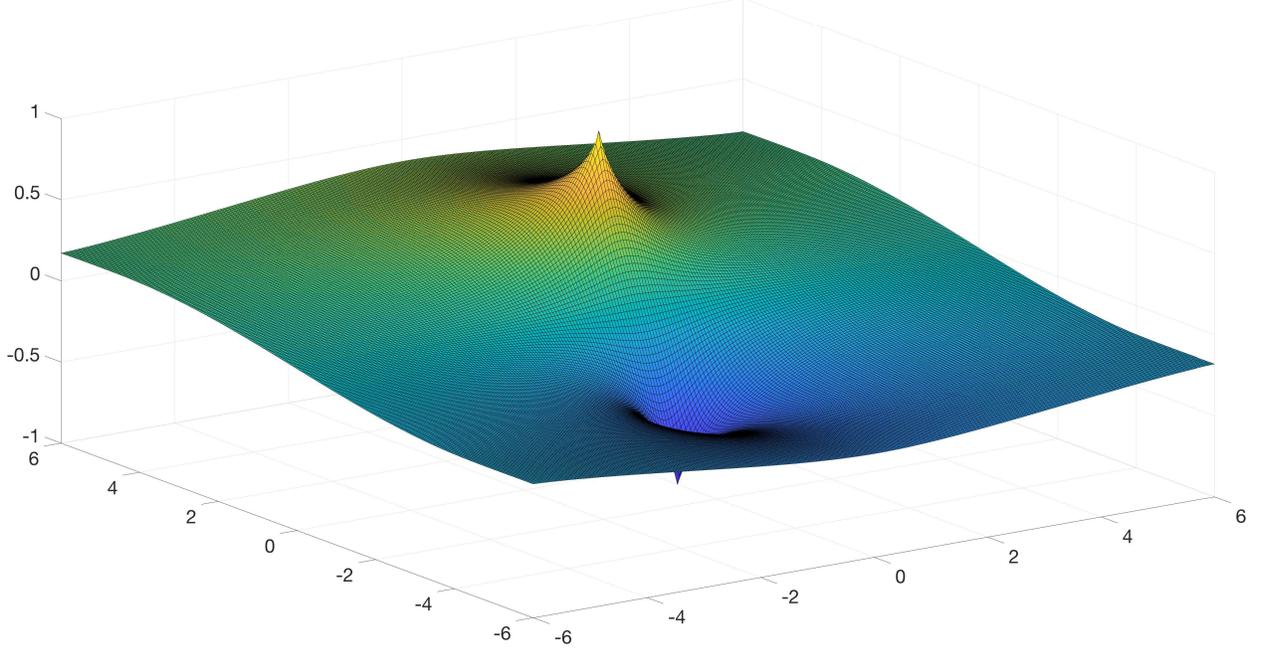}
 \caption{The graph of a numerically approximated extremal $u$ with $n=2, p=4, x_0=(0,1), y_0=(0,-1)$, $u(x_0)=1$ and $u(y_0)=-1$. Note that $u(x)\approx \frac{1}{2}(u(x_0)+u(y_0))=0$ for larger values of $|x|$.}\label{DipoleFig}
\end{figure}

\par In proving Theorem \ref{mainThm}, we will first verify that any bounded $p$-harmonic function $u$ on the exterior domain
$$
\R^n\setminus \overline{B_1}=\{x\in \R^n: |x|>1\}
$$
is asymptotically flat for $p>n\ge 2$. That is, there is some $\beta\in \R$ for which
$$
\beta=\lim_{|x|\rightarrow\infty}u(x).
$$
By employing a Harnack inequality, we can quantify this assertion and show there are positive numbers $A$ and $\alpha$ such that
\be
|u(x)-\beta|\le \frac{A\|u\|_{\infty}}{|x|^\alpha}, \quad |x|\ge 1.
\ee 
In particular, we will be able to conclude that the limit \eqref{FlatMorrey} occurs with an (at least) algebraic rate of convergence. 

\par The precise decay estimate we derive is described as follows. 
\begin{thm}\label{DecayThm}
Suppose $n\ge 2$ and $p>n$. There are positive constants $\alpha>0$ and $A>0$ such that 
$$
\sup\Big\{|u(x)-u(y)|:|x|, |y|\ge r\Big\} \le \frac{A \|u\|_\infty}{r^\alpha}, \quad r\ge 1
$$
for each function $u$ that is bounded and $p$-harmonic in $\R^n\setminus \overline{B_1}$.
\end{thm}

\par Then we'll show how these results extend to solutions $u:\R^n\rightarrow \R$ of the multipole equation
\be
-\Delta_pu=\sum^N_{i=1}c_i\delta_{x_i},
\ee
where  $x_1,\dots, x_N\in \R^n$ are distinct and $c_1,\dots, c_N\in \R$ satisfy $\sum^N_{i=1}c_i=0$. The main point is to establish that each solution $u$ is bounded. 
Moreover, we will argue that each solution  $u$ is not differentiable at any $x_i$ in which it has a strict local maximum or minimum.  Finally, 
in the appendix, we will explain the numerical method we used to produce Figure \ref{DipoleFig} as shown above.

%%%%%%%%%%%%%%%%%%%%%%%%%%%%%%%%%%%%%
\section{Bounded $p$-harmonic functions on exterior domains}
In what follows, we will suppose that 
$$
n\ge 2\quad \text{and}\quad p>n
$$
are fixed. Even though we are primarily interested in functions defined on $\R^n$, we will 
also consider functions defined on bounded domains $\Omega$ or possibly on the complement of such 
subsets.  Recall that each function in the Sobolev space 
$W^{1,p}(\Omega)$ has a $1-n/p$ H\"older continuous representative (Theorem 5 section 5.6 of \cite{MR2597943}).  Consequently, we will always identify a $W^{1,p}(\Omega)$ function with its continuous 
representative and consider $W^{1,p}(\Omega)$ as a subset of the continuous functions on $\Omega$.

\par For a given domain $\Omega\subset \R^n$, we will say that $u$ is {\it $p$-harmonic} in $\Omega$  and
write 
$$
-\Delta_pu=0\;\; \text{in}\;\;\Omega
$$
so long as $u\in W^{1,p}_{\text{loc}}(\Omega)$ and
\be
\int_{\Omega}|Du|^{p-2}Du\cdot D\phi dx=0
\ee 
for each $\phi\in C^\infty_c(\Omega)$. 
 Likewise, for a signed Borel measure $\rho$ on $\Omega$, we say that  
$$
-\Delta_pu=\rho\;\; \text{in}\;\;\Omega
$$
provided $u\in W^{1,p}_{\text{loc}}(\Omega)$ and
\be
\int_{\Omega}|Du|^{p-2}Du\cdot D\phi dx=\int_{\Omega}\phi d\rho
\ee 
for all $\phi\in C^\infty_c(\Omega)$. 

\par In this section, we will establish three facts about bounded $p$-harmonic functions on $\R^n\setminus \overline{B_1}$.
We first show that these functions are all asymptotically flat and their gradients tend to zero as $|x|\rightarrow\infty$ at a certain rate. Then we show that if one of these functions lies strictly between two values, its limit as $|x|\rightarrow\infty$ lies strictly between these two values, as well. Finally, we establish decay and monotonicity properties of two integral quantities involving these functions.

%%%%%%%%%%%%%%%%%%%%%%%%%%%%%%%%%%%%%%%%%%%%%%%%%%%%%%
\subsection{Asymptotic flatness}
As mentioned above, our first order of business is to verify the asymptotic flatness of bounded $p$-harmonic functions on $\R^n\setminus \overline{B_1}$. This is the central goal of this subsection. We also note that the first part of following statement has essentially been verified by Serrin \cite{MR0186903}, who showed that a positive $p$-harmonic function on an exterior domain has a positive limit as $|x|\rightarrow\infty$ or tends to $\infty$ at a specific rate; this result was also extended recently by Fraas and Pinchover 
\cite{MR2807111, MR2997363}.  Our result is not as general, however our proof is simple and direct.

\begin{prop}\label{pHarmonicProp}
Suppose $u$ is a bounded $p$-harmonic function on $\R^n\setminus \overline{B_1}$.  Then the limit
$$
\lim_{|x|\rightarrow \infty}u(x)
$$
exists and
$$
\lim_{|x|\rightarrow \infty}|x||Du(x)|=0.
$$
\end{prop}

\par To this end, we will need to make use of a version of Caccioppoli's inequality and a Liouville-type assertion for $p$-harmonic functions 
on punctured domains. 
% Caccioppoli
\begin{lem}
Suppose $\Omega\subset \R^n$ is a domain and $x_0\in \Omega$. Further assume $u$ satisfies 
\be
-\Delta_pu=c\delta_{x_0}
\ee
 in $\Omega$ for some constant $c$. 
Then for each nonnegative
$\zeta\in C^\infty_c(\Omega)$, 
\be\label{CaccIneq}
\int_{\Omega}\zeta^p|Du|^pdx\le p^p\int_{\Omega}|u-u(x_0)|^p|D\zeta|^pdx.
\ee 
\end{lem}
\begin{proof}
Observe
$$
\int_{\Omega}|Du|^{p-2}Du\cdot D\phi dx=c\phi(x_0)
$$
for $\phi\in W^{1,p}_0(\Omega)$. Let $\phi=\zeta^p(u-u(x_0))$ and note $\phi(x_0)=0$ and
$$
D\phi=p\zeta^{p-1}D\zeta\; (u-u(x_0)) +\zeta^p Du. 
$$
Substituting this test function above gives 
\begin{align*}
\int_{\Omega}\zeta^p|Du|^pdx&=-p\int_{\Omega}\zeta^{p-1}|Du|^{p-2}Du\cdot (u-u(x_0))D\zeta dx\\
&\le p\int_{\Omega}(\zeta|Du|)^{p-1}(|u-u(x_0)| |D\zeta|) dx\\
&\le p\left(\int_{\Omega}\zeta^p|Du|^pdx\right)^{1-1/p}\left(\int_\Omega |u-u(x_0)|^p |D\zeta|^pdx\right)^{1/p}
\end{align*}
which is \eqref{CaccIneq}. 
\end{proof}

\begin{cor}
Suppose $\Omega$ is a domain and $B_{2r}(x_0)\subset\Omega$. Further assume $u$ satisfies 
\be
-\Delta_pu=c\delta_{x_0}
\ee
 in $\Omega$ for some constant $c$. Then 
\be\label{CaccIneqLoc}
\int_{B_r(x_0)}|Du|^pdx\le \left(\frac{2p}{r}\right)^p\int_{B_{2r}(x_0)}|u-u(x_0)|^pdx.
\ee 
\end{cor}
\begin{proof}
Choose $\varphi\in C^\infty_c(B_2(0))$ with $0\le \varphi\le 1$, $\varphi\equiv 1$ in $B_1(0)$ and 
$$
\|D\varphi\|_\infty\le 2. 
$$
Then set
$$
\zeta(x)=\varphi\left(\frac{x-x_0}{r}\right), \quad x\in B_{2r}(x_0).
$$
Clearly, $\zeta\in C^\infty_c(B_{2r}(x_0))$ is nonnegative, $\zeta\equiv 1$ in $B_r(x_0)$ and 
$$
\|D\zeta\|_\infty\le \frac{2}{r}. 
$$
The conclusion follows from substituting this $\zeta$ in \eqref{CaccIneq}.
\end{proof}
% Liouville
\begin{cor}\label{LiouvilleProp}
Suppose $u$ is bounded and satisfies 
\be
-\Delta_pu=c\delta_{x_0}
\ee
in $\R^n$ for some constant $c$. Then $u$ is necessarily constant and $c=0$.  
\end{cor}
\begin{proof}
In view of \eqref{CaccIneqLoc}, 
\begin{align*}
\int_{B_r(x_0)}|Du|^pdx&\le \left(\frac{2p}{r}\right)^p\int_{B_{2r}(x_0)}|u-u(x_0)|^pdx\\ 
&\le \left(\frac{2p}{r}\right)^p(2\|u\|_\infty)^p\omega_n (2r)^n\\
&\le \frac{(4p\|u\|_\infty)^p\omega_n2^n}{r^{p-n}}
\end{align*}
for each $r>0$; here $\omega_n$ is the Lebesgue measure of $B_1$. Sending $r\rightarrow\infty$ forces $|Du|$ to vanish on $\R^n$.  
\end{proof}

We are now ready to employ these observations to fashion a proof of Proposition \ref{pHarmonicProp}.
\begin{proof}[Proof of Proposition \ref{pHarmonicProp}] 1. For $t>0$, set
\be
v_t(x):=u(tx), \quad x\in \R^n. 
\ee
Note that $v_t$ is $p$-harmonic on $\R^n\setminus \overline{B_{1/t}}$. Without loss of generality,  suppose $|u(y)|\le 1$ for all $|y|>1$, so that
$$
|v_t(x)|\le 1
$$
for $|x|>1/t$.   We will now proceed to send $t\rightarrow \infty$.

\par By a result of Ural'ceva \cite{Ural} (see also Lewis \cite{Lew} and Evans \cite{Evans}), there is $\gamma\in (0,1)$ depending on $p$ and $n$ such that 
$$
\|v_t\|_{C^{1,\gamma}(K)}\le A
$$
for each compact $K\subset \R^n\setminus\{0\}$ and $t$ sufficiently large. Here $A$ depends on $p$ and $n$ and $K$. Consequently, there is a sequence  $(v_{t_k})_{k\in \N}$ with $t_k\rightarrow\infty$ 
and $v_\infty\in C^{1}_{\text{loc}}(\R^n\setminus\{0\})$ such that  
\be
v_{t_k}\rightarrow v_\infty \;\;\text{in}\;\; C^{1}(K)
\ee
for each compact $K\subset \R^n\setminus\{0\}$. It follows easily that $v_\infty$ is $p$-harmonic on 
$\R^n\setminus\{0\}$. 

\par By Theorem 1.1 and Remark 1.6 of \cite{KicVer} (see also \cite{MR886428}), there is a constant $\mu\in \R$ such that 
\be\label{vinftyEquation}
-\Delta_p v_\infty=|\mu|^{p-2}\mu n\omega_n\delta_0
\ee
in $\R^n$. Moreover, 
$$
\lim_{|x|\rightarrow 0}\frac{|Dv_\infty(x)|}{|x|^{\left(\frac{p-n}{p-1}\right)-1}}=|\mu|.
$$
This limit gives that $|Dv_\infty|^p$ is locally integrable in a neighborhood of $0$. Since 
$$
|v_\infty(x)|\le 1
$$
for all $x\in \R^n$, we have $v_\infty\in W^{1,p}_{\text{loc}}(\R^n)$.  Corollary \ref{LiouvilleProp} then implies that $v_\infty$ is identically equal to a constant $\beta$ and so 
$$
\lim_{k\rightarrow\infty}v_{t_k}(x)=\beta
$$
locally uniformly on $\R^n\setminus\{0\}$.

\par 2. Consider 
$$
m(t):=\min_{|y|=t}u(y)
$$
for $t>1$. By the comparison principle for $p$-harmonic functions,  
$$
u(z)\ge \min\{m(t),m(s)\}
$$
for $1<s< |z|<t$. It follows that 
$$
m(\lambda t+(1-\lambda)s)\ge \min\{m(t),m(s)\}
$$
for $\lambda\in [0,1]$. In particular, $m: (1,\infty)\rightarrow [-1,1]$ is quasiconcave. So there is $r_1\ge 1$ for 
which $m|_{(r_1,\infty)}$ is monotone  (Theorem 17 in Chapter 3 of \cite{MR0496692}) and thus 
$$
\lim_{t\rightarrow \infty}m(t)=\lim_{t\rightarrow \infty}\min_{|y|=t}u(y)=\lim_{t\rightarrow \infty}\min_{|x|=1}v_t(x)
$$
exists. 

\par We can choose an $x_t\in \R^n$ with $|x_t|=1$ so that 
$$
\min_{|x|=1}v_t(x)=v_t(x_t).
$$
We may as well also suppose that $(x_{t_k})_{k\in \N}$ is convergent. In this 
case, 
$$
\lim_{t\rightarrow \infty}\min_{|x|=1}v_t(x)=\lim_{k\rightarrow \infty}\min_{|x|=1}v_{t_k}(x)=\lim_{k\rightarrow \infty}v_{t_k}(x_{t_k})=\beta.
$$
With virtually the same argument, we find 
$$
\lim_{t\rightarrow \infty}\max_{|x|=1}v_t(x)=\beta.
$$
Consequently, 
$$
\lim_{t\rightarrow\infty}v_{t}(x)=\beta
$$
uniformly for $|x|=1$.

\par 3. Now let $(y_k)_{k\in \N}\subset \R^n$ be a sequence such that $|y_k|\rightarrow\infty$. Without loss of generality, we will suppose $|y_k|>0$ and that $(y_k/|y_k|)_{k\in \N}$ is convergent as these properties are true for a subsequence of $(y_k)_{k\in \N}$.  Then 
$$
\lim_{k\rightarrow\infty}u(y_k)=\lim_{k\rightarrow\infty}u\left(|y_k|\frac{y_k}{|y_k|}\right)
=\lim_{k\rightarrow\infty}v_{|y_k|}\left(\frac{y_k}{|y_k|}\right)=\beta,
$$
and we conclude that 
$$
\lim_{|y|\rightarrow \infty}u(y)=\beta.
$$

\par We also have that 
$$
Dv_t(x)=Du(tx)t
$$
tends to $0\in \R^n$ uniformly for $|x|=1$. Choosing $(y_k)_{k\in \N}$ as above, we find
\begin{align*}
\lim_{k\rightarrow \infty}|y_k||Du(y_k)|&=\lim_{k\rightarrow \infty}|y_k|\left|Du\left(|y_k|\frac{y_k}{|y_k|}\right)\right|\\
&=\lim_{k\rightarrow \infty}\left|Dv_{|y_k|}\left(\frac{y_k}{|y_k|}\right)\right|\\
&=0.
\end{align*}
That is, 
$$
\lim_{|y|\rightarrow \infty}|y||Du(y)|=0.
$$
\end{proof}
\begin{rem}
This theorem can be proved without appealing to the $C^{1,\gamma}_{\text{loc}}$ estimates for $p$-harmonic functions.  Local uniform convergence of a subsequence of $(v_t)_{t>0}$ in $\R^n\setminus\{0\}$ would follow from Morrey's inequality, and convergence in $W^{1,p}_{\text{loc}}(\R^n\setminus\{0\})$ can be verified using the Browder and Minty method (as described in section 9.1 of \cite{MR2597943}).
\end{rem}
\begin{rem}
In Corollary \ref{WeakCompCor} below, we will show that $\min_{|x|=r}u(x)$ is nondecreasing and $\max_{|x|=r}u(x)$ is nonincreasing for all $r\in (1,\infty)$.  
\end{rem}

%----------------------------------------------------------------------------------------------
\subsection{Strict bounds on limiting values}
The next assertion states that the limit of a bounded $p$-harmonic function on $\R^n\setminus \overline{B_1}$
always lies strictly within the bounds observed by the function. In particular, any bounded and positive $p$-harmonic function on an exterior domain has a positive limit.  Pinchover and Tintarev \cite{MR2404043} established this conclusion using a different argument and for more general operators. 
\begin{prop}\label{LimitPosPreserving}
Suppose  $u$ is $p$-harmonic in $\R^n\setminus \overline{B_1}$ and 
$$
a< u(x)<b, \quad x\in \R^n\setminus \overline{B_1}
$$
for some $a,b\in \R$. Then
\be\label{betaInside}
a<\lim_{|x|\rightarrow \infty}u(x)<b.
\ee
\end{prop}
\begin{proof}
Fix $r>1$, and for $R>r$ define
$$
w_R(x)=\frac{R^{\frac{p-n}{p-1}}-|x|^{\frac{p-n}{p-1}}}{R^{\frac{p-n}{p-1}}-r^{\frac{p-n}{p-1}}}, \quad r\le |x|\le R.
$$
Note that $w_R$ is $p$-harmonic in the annulus $B_R\setminus{\overline{B_{r}}}$,
$$
w_R|_{\partial B_{r}}=1\quad \text{and}\quad w_R|_{\partial B_R}=0.
$$
Now choose $\delta>0$ such that 
$$
\min_{x\in\partial B_{r}}u(x)-a\ge \delta. 
$$
By comparison, 
\be
u(x)-a\ge \delta w_R(x), \quad \quad r\le |x|\le R.
\ee
\par Let $e_1=(1,0,\dots, 0)$ and suppose $R>2r$. Then $r<\frac{1}{2}R<R$ and so
\begin{align}
u\left(\frac{R}{2}e_1\right)&\ge a+\delta w_R\left(\frac{R}{2}e_1\right)\\
&= a+\delta \frac{R^{\frac{p-n}{p-1}}-(R/2)^{\frac{p-n}{p-1}}}{R^{\frac{p-n}{p-1}}-r^{\frac{p-n}{p-1}}}\\
&= a+\delta \frac{1-(1/2)^{\frac{p-n}{p-1}}}{1-(r/R)^{\frac{p-n}{p-1}}}.
\end{align}
As a result, 
$$
\lim_{|x|\rightarrow \infty}u(x)=\lim_{R\rightarrow \infty}u\left(\frac{R}{2}e_1\right)\ge 
a+\delta \left(1-(1/2)^{\frac{p-n}{p-1}}\right)>a.
$$
Likewise, we find $\lim_{|x|\rightarrow \infty}u(x)<b$.
\end{proof}
\begin{rem}
We will see in Corollary \ref{WeakCompCor}, that the same conclusion holds only assuming 
$$
a<u(x)<b,\quad |x|=r
$$ 
for some $r>1$. This improvement relies on a global comparison property of bounded $p$-harmonic functions on the exterior domain $\R^n\setminus{\overline{B_1}}$. 
\end{rem}

%--------------------------------------------------------------
\subsection{Integral decay and monotonicity}
In Proposition \ref{pHarmonicProp}, we showed that if $u$ is a bounded $p$-harmonic function in $\R^n\setminus \overline{B_1}$, then 
\be\label{GradientLimitPLap}
\lim_{|x|\rightarrow \infty}|x||Du(x)|=0.
\ee
This limit immediately implies the following decay property.

\begin{cor}\label{CorBoundedInt}
Suppose  $u$ is bounded and $p$-harmonic in $\R^n\setminus \overline{B_1}$. 
Then 
\be
\int_{|x|>s}|Du|^pdx<\infty
\ee
for any $s>1$. Moreover, 
\be\label{pHarmonicDecay}
\lim_{r\rightarrow\infty}r^{p-n}\int_{|x|>r}|Du|^pdx=0.
\ee
\end{cor}
\begin{proof}
Fix $\epsilon>0$. By \eqref{GradientLimitPLap}, there is $r>s$ so large that 
$$
|Du(x)|\le \frac{\epsilon}{|x|}
$$
for $|x|\ge r$. Then
\begin{align*}
\int_{|x|>r}|Du|^pdx&\le\epsilon^p n\omega_n\int^\infty_r\tau^{-p}\tau^{n-1}d\tau=\epsilon^p\frac{n\omega_n}{(p-n)r^{p-n}}.
\end{align*}
Since 
$$
\int_{s<|x|<r}|Du|^pdx<\infty,
$$
the first assertion follows.  As for the second claim,
$$
\lim_{r\rightarrow\infty}r^{p-n}\int_{|x|>r}|Du|^pdx\le \epsilon^p\frac{n\omega_n}{(p-n)}.
$$
The conclusion follows as $\epsilon>0$ is arbitrary. 
\end{proof}

Using a certain identity for smooth $p$-harmonic functions, we can strengthen the conclusion of the previous corollary. 
\begin{prop}\label{pHarmonicProp2}
Suppose  $u$ is smooth, bounded and $p$-harmonic in $\R^n\setminus \overline{B_1}$. 
Then 
\be\label{pHarmonicDecay2}
(1,\infty)\ni r\mapsto r^{p-n}\int_{|x|>r}|Du|^pdx
\ee
is nonincreasing. In particular, 
\be\label{MonConpLapInt}
\lim_{r\rightarrow\infty}r^{p-n}\int_{|x|>r}|Du|^pdx=\inf_{r>1}r^{p-n}\int_{|x|>r}|Du|^pdx=0.
\ee
Moreover, 
\be\label{MonConpLapInt2}
r^{p-n}\int_{|x|>r}|Du|^pdx=p\int_{|x|>r}|x|^{p-n}|Du|^{p-2}\left(Du\cdot \frac{x}{|x|}\right)^2dx
\ee
for each $r>1$.
\end{prop}
\begin{proof}
As $u$ is smooth, direct computation gives  
\be\label{IdentityPLap}
\text{div}\left(\left(Du\cdot x +\left(\frac{n}{p}-1\right)u\right)p|Du|^{p-2}Du-|Du|^px\right)=0
\ee
in $\R^n\setminus \overline{B_1}$ (Chapter 8 section 6 of \cite{MR2597943}). Integrating both sides of \eqref{IdentityPLap} over 
$r<|x|<R$ gives
\begin{align}\label{FirstIntByPartsForMon}
0&=\int_{r<|x|< R}\text{div}\left(\left(Du\cdot x +\left(\frac{n}{p}-1\right)u\right)p|Du|^{p-2}Du-|Du|^px\right)dx\nonumber\\
&=\int_{|x|=R}\left(\left(Du\cdot x +\left(\frac{n}{p}-1\right)u\right)p|Du|^{p-2}Du-|Du|^px\right)\cdot \frac{x}{R}d\sigma\nonumber\\
&\quad -\int_{|x|=r}\left(\left(Du\cdot x +\left(\frac{n}{p}-1\right)u\right)p|Du|^{p-2}Du-|Du|^px\right)\cdot \frac{x}{r}d\sigma\nonumber\\
&=-R\int_{|x|=R}\left(|Du|^p-p|Du|^{p-2}(\partial_ru)^2\right)d\sigma+(n-p)\int_{|x|=R}u|Du|^{p-2}Du\cdot\frac{x}{R}d\sigma\\
&\quad +r\int_{|x|=r}\left(|Du|^p-p|Du|^{p-2}(\partial_ru)^2\right)d\sigma+(n-p)\int_{|x|=r}u|Du|^{p-2}Du\cdot\frac{-x}{r}d\sigma.\nonumber
\end{align}
Here 
$$
\partial_ru(x):=Du(x)\cdot \frac{x}{|x|}
$$
is the radial derivative of $u$ and $\sigma$ is $n-1$ dimensional Hausdorff measure. 
\par In view of \eqref{GradientLimitPLap}, 
$$
-R\int_{|x|=R}\left(|Du|^p-p|Du|^{p-2}(\partial_ru)^2\right)d\sigma+(n-p)\int_{|x|=R}u|Du|^{p-2}Du\cdot\frac{x}{R}d\sigma=o(R^{n-p})$$
as $R\rightarrow \infty$.  So we can send $R\rightarrow \infty$ in \eqref{FirstIntByPartsForMon} to conclude  
\begin{align*}
0&=r\int_{|x|=r}\left(|Du|^p-p|Du|^{p-2}(\partial_ru)^2\right)d\sigma+(n-p)\int_{|x|=r}u|Du|^{p-2}Du\cdot\frac{-x}{r}d\sigma\\
&=r\int_{|x|=r}\left(|Du|^p-p|Du|^{p-2}(\partial_ru)^2\right)d\sigma+(n-p)\int_{|x|>r}|Du|^{p}dx.
\end{align*}

\par Now observe
\begin{align}\label{MonFormulaInt}
\frac{d}{dr}\left\{r^{p-n}\int_{|x|>r}|Du|^pdx\right\}&=(p-n)r^{p-n-1}\int_{|x|>r}|Du|^pdx-r^{p-n}\int_{|x|=r}|Du|^pd\sigma\nonumber\\
&=r^{p-n-1}\left\{(p-n)\int_{|x|>r}|Du|^pdx - r\int_{|x|=r}|Du|^pd\sigma\right\}\nonumber\\
&=r^{p-n-1}\left\{-rp\int_{|x|=r}|Du|^{p-2}(\partial_ru)^2d\sigma\right\}\nonumber\\
&=-pr^{p-n}\int_{|x|=r}|Du|^{p-2}(\partial_ru)^2d\sigma.
\end{align}
As a result, 
$$
(1,\infty)\ni r\mapsto r^{p-n}\int_{|x|>r}|Du|^pdx
$$
is nonincreasing. This quantity tends to $0$ as $r\rightarrow \infty$ by the previous corollary, so we conclude \eqref{MonConpLapInt} 
by monotone convergence.  Integrating the monotonicity formula \eqref{MonFormulaInt} from $r=s$ to $r=\infty$  gives 
\begin{align}
s^{p-n}\int_{|x|>s}|Du|^pdx&=p\int^\infty_sr^{p-n}\int_{|x|=r}|Du|^{p-2}(\partial_ru)^2d\sigma dr \noindent\\
&=p\int^\infty_s\int_{|x|=r}|x|^{p-n}|Du|^{p-2}(\partial_ru)^2d\sigma dr\noindent \\
&=p\int_{|x|>s}|x|^{p-n}|Du|^{p-2}(\partial_ru)^2dx
\end{align}
which is \eqref{MonConpLapInt2}. 
\end{proof}

%%%%%%%%%%%%%%%%%%%%%%%%%%%%%%%%%%%%%
\section{Asymptotics of extremals}
This section is dedicated to the proof of Theorem \ref{mainThm}. Let $u$ be an extremal satisfying \eqref{HolderMax}. In Proposition 3.5 of \cite{HyndSeuf}, we established that 
\be\label{MorreyBounds}
\min\{u(x_0),u(y_0)\}\le u(x)\le \max\{u(x_0),u(y_0)\}
\ee
for each $x\in \R^n$; this inequality is also established in Lemma \ref{Boundedai} below. As a result, $u$ is uniformly bounded and is $p$-harmonic 
in $\R^n\setminus \overline{B_s}$ for 
$$
s:= \max\{|x_0|,|y_0|\}.
$$
It follows from Proposition \ref{pHarmonicProp} that the limit
$$
\lim_{|x|\rightarrow \infty}u(x)
$$
exists and 
$$
\lim_{|x|\rightarrow \infty}|x||Du(x)|=0.
$$
As $u$ is smooth in $\R^n\setminus \overline{B_s}$  (section 4.3 of  \cite{HyndSeuf}), we can apply Proposition \ref{pHarmonicProp2}
to conclude
$$
r^{p-n}\int_{|x|>r}|Du|^pdx=\int_{|x|>r}|x|^{p-n}|Du|^{p-2}\left(Du\cdot \frac{x}{|x|}\right)^2dx
$$
for $r>s$. Moreover, this quantity is nonincreasing on $(s,\infty)$ and tends to $0$ as $r\rightarrow\infty$. 

\par In Proposition 3.4 of \cite{HyndSeuf}, we showed
$$
u\left(x -2\frac{\left((x_0-y_0)\cdot (x-\frac{1}{2}(x_0+y_0)\right)}{|x_0-y_0|^2}(x_0-y_0)\right)-\frac{u(x_0)+u(y_0)}{2}
=-\left(u(x)-\frac{u(x_0)+u(y_0)}{2}\right)
$$
for each $x\in \R^n$. This equality implies that $u-\frac{1}{2}(u(x_0)+u(y_0))$ is antisymmetric with respect to reflection about the hyperplane 
$$
\Pi:=\left\{x\in \R^n: (x_0-y_0)\cdot \left(x-\frac{1}{2}(x_0+y_0)\right)=0\right\}.
$$
In particular, 
$$
u(x)=\frac{1}{2}(u(x_0)+u(y_0))
$$
for each $x\in \Pi$. As $\Pi$ is unbounded, it must be that 
$$
\lim_{|x|\rightarrow \infty}u(x)=\frac{1}{2}(u(x_0)+u(y_0)).
$$
\begin{figure}[h]
\centering
 \includegraphics[width=1\columnwidth]{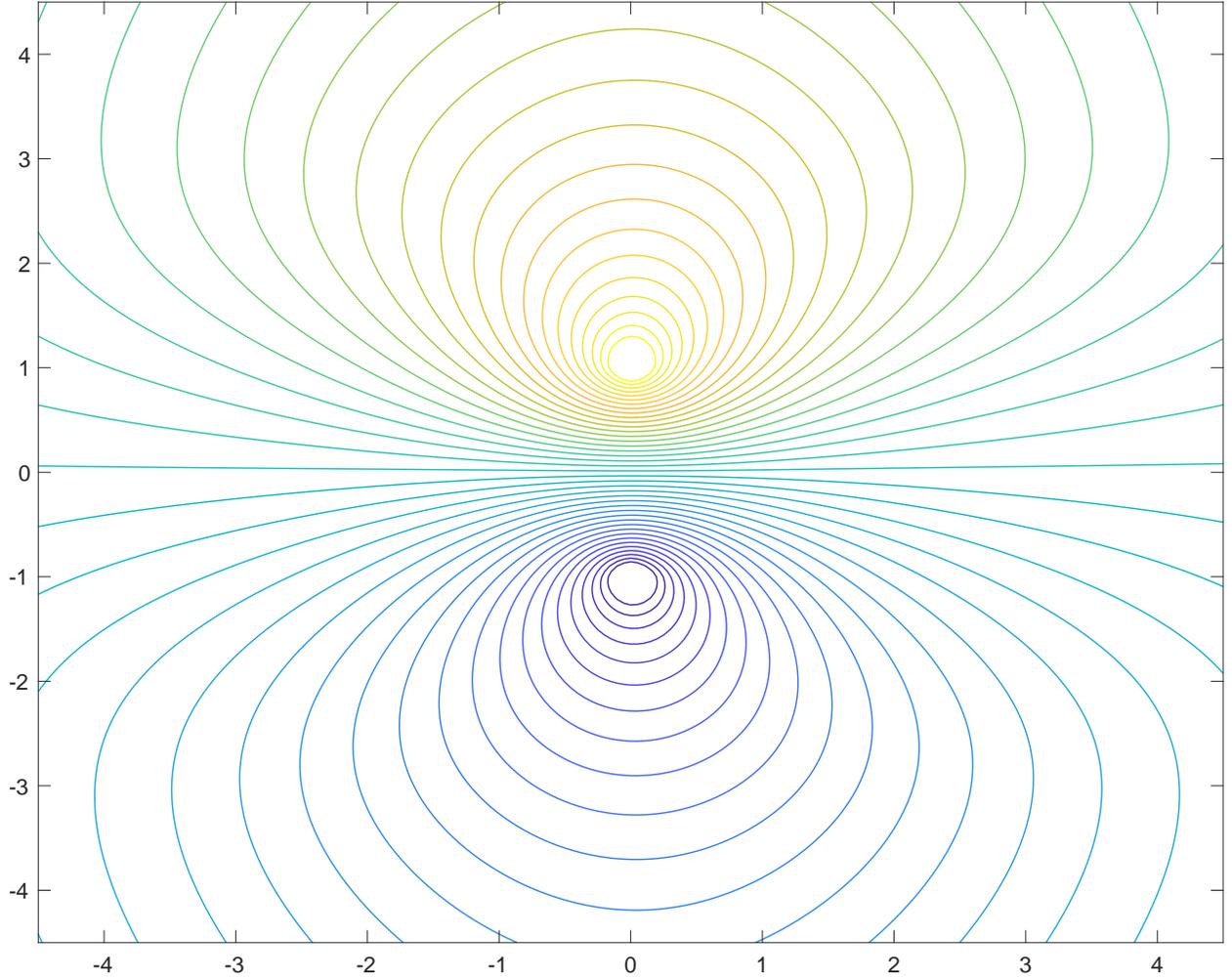}
 \caption{Level sets of the approximate extremal computed for Figure \ref{DipoleFig}. Each level set except the line $x_2=0$ bounds a convex, compact subset of $\R^2$.}\label{ContourPlot}
\end{figure}

\begin{rem}
If $u$ is an extremal which satisfies 
\be
\sup_{x\ne y}\left\{\frac{|u(x)-u(y)|}{|x-y|^{1-n/p}}\right\}=\frac{u(x_0)-u(y_0)}{|x_0-y_0|^{1-n/p}}>0
\ee
for distinct $x_0,y_0\in \R^n$,  
\be
\{x\in \R^n: u(x)\ge t\}\quad\text{and}\quad \{x\in \R^n: u(x)\le s\}
\ee
are convex for
$$
\frac{u(x_0)+u(y_0)}{2}<t\le u(x_0)\quad\text{and}\quad\frac{u(x_0)+u(y_0)}{2}>s\ge u(y_0),
$$
respectively. This was proved in Proposition 4.4 of \cite{HyndSeuf}.  An immediate corollary of Theorem \ref{mainThm} is that these subsets are {\it compact}, as displayed in Figure \ref{ContourPlot}. 
\end{rem}

%%%%%%%%%%%%%%%%%%%%%%%%%
\section{Quantitative flatness}
We will now establish a Harnack inequality for bounded, nonnegative $p$-harmonic functions on $\R^n\setminus \overline{B_1}$. We will then prove Theorem \ref{DecayThm} similar to how H\"older continuity of $p$-harmonic functions can be established with a Harnack inequality (as explained in section 2 of \cite{MR2242021}).  To this end, we will start with the following comparison principle. 
\begin{lem}
Suppose $r>1$ and that $u,v $ are bounded and $p$-harmonic in $\R^n\setminus \overline{B_1}$ with 
$$
u\le v
$$
on $\partial B_r$. Then 
$$
u\le v
$$
in $\R^n\setminus{\overline{B_{r}}}$. 
\end{lem}
\begin{proof}
In view of the monotonicity of the mapping $z\mapsto |z|^{p-2}z$, 
\begin{align}\label{ComparisonInitialEstimate}
0&\le \int_{\{u>v\}\cap\{|x|>r\} }(|Du|^{p-2}Du-|Dv|^{p-2}Dv)\cdot (Du-Dv)dx \noindent\\
&=  \int_{|x|>r}(|Du|^{p-2}Du-|Dv|^{p-2}Dv)\cdot D(u-v)^+dx.
\end{align}
As $(u-v)^+$ is bounded and vanishes on $\partial B_r$, we can integrate by parts and appeal to  \eqref{GradientLimitPLap} in order to deduce
\begin{align}\label{ComparisonSecondEstimate}
 &\int_{|x|>r}(|Du|^{p-2}Du-|Dv|^{p-2}Dv)\cdot D(u-v)^+dx\\
 &\hspace{1in}=\lim_{R\rightarrow\infty} \int_{r<|x|<R}(|Du|^{p-2}Du-|Dv|^{p-2}Dv)\cdot D(u-v)^+dx \\
 &\hspace{1in}=\lim_{R\rightarrow\infty} \int_{|x|=R}(u-v)^+(|Du|^{p-2}Du-|Dv|^{p-2}Dv)\cdot \frac{x}{R}dx\\
 &\hspace{1in}=\lim_{R\rightarrow\infty} o(R^{n-p})\\
 &\hspace{1in}=0.
\end{align}
Combining with \eqref{ComparisonInitialEstimate} gives
$$
 \int_{\{u>v\}\cap\{|x|>r\} }(|Du|^{p-2}Du-|Dv|^{p-2}Dv)\cdot (Du-Dv)dx=0.
$$

\par As $z\mapsto |z|^{p-2}z$ is strictly monotone, it must either be that the Lebesgue measure of ${\{u>v\}\cap\{|x|>r\} }$ is zero or that $Du=Dv$ almost everywhere in ${\{u>v\}\cap\{|x|>r\} }$. If the Lebesgue measure of ${\{u>v\}\cap\{|x|>r\} }$ is zero, $u(x)\le v(x)$ for almost every $|x|>r$; as $u,v$ are continuous, this would imply that $u(x)\le v(x)$ for every $|x|>r$. Otherwise, $D(u-v)^+=0$ in $\R^n\setminus \overline{B_r}$ which would mean  $(u-v)^+$ is constant throughout $\R^n\setminus \overline{B_r}.$ Since 
$(u-v)^+$ vanishes on $\partial B_r$, we would have $(u-v)^+\equiv 0$ in $\R^n\setminus \overline{B_r}$. That is, $u\le v$ in $\R^n\setminus \overline{B_r}$.
\end{proof}
\begin{cor}\label{WeakCompCor}
Suppose $u$ is $p$-harmonic and bounded in $\R^n\setminus \overline{B_1}$.

\begin{enumerate}[(i)]

\item For each $r> 1$,
$$
\sup_{|x|\ge r}u(x)=\sup_{|x|=r}u(x)\quad\text{and}\quad \inf_{|x|\ge r}u(x)=\inf_{|x|=r}u(x).
$$

\item For $1< r_1\le r_2$,
$$
\sup_{|x|= r_1}u(x)\ge \sup_{|x|=r_2}u(x)\quad\text{and}\quad \inf_{|x|= r_1}u(x)\le \inf_{|x|=r_2}u(x).
$$

\item For $1< r_1\le r_2$,
$$
\sup_{|x|= r_1}u(x)= \sup_{r_1\le|x|\le r_2}u(x)\quad\text{and}\quad \inf_{|x|= r_1}u(x)= \inf_{r_1\le|x|\le r_2}u(x).
$$

\item If $r>1$ and $a<u(x)<b$ for $|x|=r$, then 
$$
a<\lim_{|x|\rightarrow\infty}u(x)<b.
$$
\end{enumerate}

\end{cor}
\begin{proof}
We will only prove the statements involving suprema.  $(i)$ Let $v$ denote the constant function on $\R^n$ which 
is equal to $\sup_{|x|=r}u(x)$. As $v$ is bounded, $p$-harmonic, and $u(y)\le v(y)$ for $|y|=r$, it follows that 
$u(y)\le v(y)$ for each $|y|\ge r$. That is, 
$$
u(y)\le \sup_{|x|=r}u(x),\quad |y|\ge r.
$$
$(ii)$ By part $(i)$, 
$$
\sup_{|x|=r_1}u(x)=\sup_{|x|\ge r_1}u(x)\ge \sup_{|x|\ge r_2}u(x)=\sup_{|x|=r_2}u(x).
$$
$(iii)$ Part $(i)$ also implies 
$$
\sup_{|x|\ge r_1}u(x)=\sup_{|x|=r_1}u(x)\le \sup_{r_1\le |x|\le r_2}u(x)\le\sup_{|x|\ge r_1}u(x). 
$$
$(iv)$ Choose $\delta>0$ so small that $\sup_{|x|=r}u(x) < b-\delta$. By part $(i)$, $u(x)<b-\delta$ for each $|x|\ge r$.  Consequently, $\lim_{|x|\rightarrow\infty}u(x)\le b-\delta<b$. 
\end{proof}

The following harnack inequality is now an easy consequence of these observations. 
\begin{prop}\label{HarnackLemma}
There is a constant $C>1$ such that 
\be\label{HarnackIN}
\sup_{|x|\ge 2r}u(x)\le C\inf_{|x|\ge 2r}u(x)
\ee
for each $r>0$ and bounded, nonnegative $p$-harmonic $u$ in $\R^n\setminus \overline{B_r}$.
\end{prop}
\begin{rem}
The example $u(x)=|x|^{\frac{p-n}{p-1}}$ shows that the boundedness assumption cannot be removed.  
\end{rem}
\begin{proof}
First suppose $r=1$ and choose $C>1$ such that 
\be\label{LocalHarnack}
\sup_{2<|x|<3}v\le C\inf_{2<|x|<3}v.
\ee
for each for each nonnegative $p$-harmonic function $v$ on $\{1<|x|<4\}$. Such a constant $C$ exists by the Harnack inequality proved by Serrin in section 5 of \cite{MR0170096}.  In view of Corollary \ref{WeakCompCor} and \eqref{LocalHarnack},
\begin{align*}
\sup_{|x|\ge 2}u(x)&= \sup_{2\le |x|\le 3}u(x) \\
&\le C\inf_{2\le |x|\le 3}u(x) \\ 
&= C\inf_{|x|\ge 2}u(x).
\end{align*}

\par For general $r>0$, we set $w(y)=u(ry)$ for $|y|>1$.  Then $w$ is bounded, nonnegative, and $p$-harmonic in $\R^n\setminus \overline{B_1}$. By our computation above, 
$$
\sup_{|y|\ge 2}w(y)\le C\inf_{|y|\ge 2}w(y)
$$
with the constant $C$ from \eqref{LocalHarnack}. It follows that \eqref{HarnackIN} holds with this constant $C$.
\end{proof}
Along with this Harnack inequality, we will need one more fact to prove Theorem \ref{DecayThm}. 
\begin{lem}\label{DecayLem}
Suppose $f: [1,\infty)\rightarrow [0,\infty)$ is nonincreasing and satisfies 
$$
f(2r)\le \mu f(r), \quad r\ge 1
$$
for some $\mu\in (0,1)$. Then 
$$
f(r)\le \frac{1}{\mu}r^{\left(\frac{\ln \mu}{\ln 2}\right)}f(1)
$$
for $r\ge 1$. 
\end{lem}
\begin{rem}
$\ln\mu<0$, so $f(r)$ decays like a power of $r$ as $r\rightarrow \infty$.
\end{rem}
\begin{proof}
By induction, 
$$
f(2^k)\le \mu^kf(1)
$$
for each nonnegative integer $k$. Choose $k\in \N$ so that 
$$
2^{k-1}\le r< 2^k
$$ 
and 
$$
k-1\le \frac{\ln r}{\ln 2}< k.
$$
Then 
\begin{align*}
f(r)&\le f(2^{k-1})\\
&\le \frac{1}{\mu}\mu^kf(1)\\
&\le \frac{1}{\mu}\mu^{\frac{\ln r}{\ln 2}}f(1)\\
&=\frac{1}{\mu}r^{\frac{\ln \mu}{\ln 2}}f(1).
\end{align*}
\end{proof}
\begin{proof}[Proof of Theorem \ref{DecayThm}]
Set 
$$
M(r):=\sup_{|x|\ge r}u(x),\quad m(r):=\inf_{|x|\ge r}u(x),\quad\text{and}\quad \omega(r):=M(r)-m(r)
$$
for $r\ge 1$.  Observe that $M(r), -m(r)$ and $\omega(r)$ are nonincreasing. Also note
$u(x)-m(r)$ is a bounded, nonnegative $p$-harmonic function for $|x|\ge r$. 
By Proposition \ref{HarnackIN},
$$
M(2r)-m(r)=\sup_{|x|\ge 2r}(u(x)-m(r))\le C\inf_{|x|\ge 2r}(u(x)-m(r))=C(m(2r)-m(r))
$$
for some $C>1$ independent of $r$. Likewise $M(r)-u(x)$ is a bounded, nonnegative $p$-harmonic function for $|x|\ge r$, so
$$
M(r)-m(2r)=\sup_{|x|\ge 2r}(M(r)-u(x))\le C\inf_{|x|\ge 2r}(M(r)-u(x))=C(M(r)-M(2r)).
$$
\par Adding these inequalities gives 
$$
\omega(2r)+\omega(r)\le C(-\omega(2r)+\omega(r)).
$$
That is, 
$$
\omega(2r)\le \frac{C-1}{C+1}\omega(r),\quad r\ge 1. 
$$
By the Lemma \ref{DecayLem}, 
$$
\omega(r)\le \frac{A\|u\|_\infty}{r^\alpha}
,\quad r\ge 1
$$
for some $\alpha,A>0$; here we used $\omega(1)\le 2\|u\|_\infty$.  In particular, 
$$
\sup\Big\{|u(x)-u(y)|:|x|, |y|\ge r\Big\} \le \frac{A \|u\|_\infty}{r^\alpha}
$$
for $r\ge 1$. 
\end{proof}
A minor variation of our proof of Theorem \ref{DecayThm} combined with \eqref{MorreyBounds} gives the following conclusion. 
\begin{cor}
Assume $u$ is an extremal which satisfies \eqref{HolderMax}. There are positive $A, \alpha,$ and $s$ such that
$$
\left|u(x)-\frac{1}{2}(u(x_0)+u(y_0))\right|\le  \frac{A\max\{|u(x_0)|,|u(y_0)|\}}{|x|^{\alpha}}
$$
for each $|x|>s$. 
\end{cor}

%%%%%%%%%%%%%%%%%%%%%%%%%%%%%%%%%%%%%%
\section{Multipole equation}
We define  
\be\label{SpaceCee}
{\cal D}^{1,p}(\R^n):=\left\{u\in L^{1}_{\text{loc}}(\R^n): u_{x_i}\in L^p(\R^n)\;\text{for $i=1,\dots, n$}\right\}
\ee
and suppose $x_1,\dots, x_N\in \R^n$ are distinct and $a_1,\dots, a_N\in \R$ are given.  Let us consider the minimization problem: find $v\in {\cal D}^{1,p}(\R^n)$ which minimizes the integral 
\be\label{pDirichlet}
\int_{\R^n}|Dv|^pdx
\ee
subject to the constraints
\be\label{MultiConstraint}
v(x_i)=a_i,\quad i=1,\dots, N. 
\ee
Direct methods from the calculus of variations can be used to show that there is a minimizer $u\in {\cal D}^{1,p}(\R^n)$. Moreover, as the Dirichlet integral \eqref{pDirichlet} is strictly convex, $u$ is unique.   

\par These observations were first noted by Kichenassamy in  section 2.3 of \cite{MR873469}.  Discrete analogs of this minimization problem also arise in semi-supervised learning with labels as studied recently by Calder \cite{MR3893728} and by Slep\v{c}ev and Thorpe \cite{MR3953458}. We became interested in this problem when we noticed that the minimizer $u$ above satisfies a generalized version of the PDE solved by Morrey extremals. 
\begin{prop}
(i) Suppose $u\in {\cal D}^{1,p}(\R^n)$ minimizes \eqref{pDirichlet} subject to the constraints \eqref{MultiConstraint}. Then there are constants $c_1,\dots, c_N\in \R$ such that 
\be\label{pMultiCond}
\int_{\R^n}|Du|^{p-2}Du\cdot D\phi dx=\sum^N_{i=1}c_i\phi(x_i)
\ee
for all $\phi\in{\cal D}^{1,p}(\R^n)$.\\
(ii) Conversely, assume $u\in {\cal D}^{1,p}(\R^n)$ satisfies \eqref{pMultiCond} and the constraints \eqref{MultiConstraint}. Then $u$ minimizes \eqref{pDirichlet} among all $v\in {\cal D}^{1,p}(\R^n)$ which satisfy \eqref{MultiConstraint}. 
\end{prop}
\begin{rem}
Choosing $\phi\equiv 1$ in \eqref{pMultiCond}, we see that  $\sum^N_{i=1}c_i=0$.
\end{rem}
\begin{rem}
If $u\in {\cal D}^{1,p}(\R^n)$ satisfies \eqref{pMultiCond}, then $u$ is a solution of the multipole equation
\be\label{pMultiPole}
-\Delta_pu=\sum^N_{i=1}c_i\delta_{x_i}
\ee
\end{rem}
\begin{proof}
$(i)$ Let $\phi\in {\cal D}^{1,p}(\R^n)$ and choose $r>0$ so small that all of the balls $B_r(x_1),\dots, B_r(x_N)$ are disjoint. It is straightforward to check that $u$ is $p$-harmonic in  $\R^n\setminus\bigcup^N_{i=1}B_r(x_i)$. 
Consequently, we can integrate by parts to find
\be\label{pMultiCond2}
\int_{\R^n\setminus\bigcup^N_{i=1}B_r(x_i) }|Du|^{p-2}Du\cdot D\phi dx =-\sum^N_{i=1}\int_{\partial B_r(x_i)}\phi |Du|^{p-2}Du\cdot \frac{x-x_i}{r}d\sigma.
\ee
By Theorem 1.1 and Remark 1.6 of \cite{KicVer},  
$$
\lim_{r\rightarrow 0}\left[-\int_{\partial B_r(x_i)}\phi |Du|^{p-2}Du\cdot \frac{x-x_i}{r}d\sigma\right]=c_i\phi(x_i)
$$
for some $c_i\in \R$ independent of $\phi$ for each $i=1,\dots, N$.  As a result, we can send $r\rightarrow 0^+$ in \eqref{pMultiCond2} and conclude \eqref{pMultiCond}. 
\par $(ii)$ Suppose $u\in {\cal D}^{1,p}(\R^n)$ fulfills \eqref{pMultiCond} and
 that $u, v\in {\cal D}^{1,p}(\R^n)$ satisfy \eqref{MultiConstraint}. As 
 $$
 |z|^p\ge |w|^p+p|w|^{p-2}w\cdot (z-w)
 $$ 
for all $z,w\in \R^n$, 
 $$
|Dv|^p\ge |Du|^p+p|Du|^{p-2}Du\cdot (Dv-Du)
 $$
holds almost everywhere in $\R^n$. Integrating this inequality gives
 \begin{align*}
\int_{\R^n}|Dv|^pdx&\ge \int_{\R^n}|Du|^pdx+p\int_{\R^n}|Du|^{p-2}Du\cdot D(v-u)dx\\
&= \int_{\R^n}|Du|^pdx+p\sum^N_{i=1}c_i(u-v)(x_i)\\
&= \int_{\R^n}|Du|^pdx.
 \end{align*}
\end{proof}
It also turns out that minimizers are uniformly bounded. 
\begin{lem}\label{Boundedai}
Suppose $u\in {\cal D}^{1,p}(\R^n)$ minimizes \eqref{pDirichlet} subject to the constraints \eqref{MultiConstraint}. Then 
\be
\min_{1\le i\le N}a_i\le u(x)\le \max_{1\le i\le N}a_i
\ee
for each $x\in \R^n$.  Moreover, if not all of the $a_i$ are identical,
\be\label{SimplyUpandLowBounds}
\min_{1\le i\le N}a_i< u(x)< \max_{1\le i\le N}a_i
\ee
for each $x\in \R^n\setminus\{x_1,\dots, x_N\}$.
\end{lem}
\begin{proof}
We will only establish the claimed upper bounds. Set $$M:=\max_{1\le i\le N}a_i$$ and define 
$$
v(x)=\min\{u(x),M\}, \quad x\in \R^n. 
$$
It is plain to see that $v\le M$ and that $v$ satisfies \eqref{MultiConstraint}. Moreover, 
$$
\int_{\R^n}|Dv|^pdx=\int_{u\le M}|Du|^pdx\le\int_{\R^n}|Du|^pdx.
$$
So $v\in {\cal D}^{1,p}(\R^n)$ minimizes \eqref{pDirichlet} subject to \eqref{MultiConstraint}. It follows that $u\equiv v\le M$. 

\par Observe that $u-M$ is nonpositive and $p$-harmonic in the domain $\R^n\setminus\{x_1,\dots, x_N\}$. By the strong 
maximum principle (Corollary 2.21 of \cite{MR2242021}), it is either 
that $u\equiv M$ or $u<M$ in $\R^n\setminus\{x_1,\dots, x_N\}$. Since $u$ is not constant in $\R^n$ is must be that 
$u<M$ in $\R^n\setminus\{x_1,\dots, x_N\}$.  
\end{proof}
The following corollary is now easily seen as a consquence of Propositions \ref{pHarmonicProp} and \ref{LimitPosPreserving}.
\begin{cor}
Suppose $u\in {\cal D}^{1,p}(\R^n)$ minimizes \eqref{pDirichlet} subject to the constraints \eqref{MultiConstraint}. Then the limit 
\be\label{MultiPoleLim}
\lim_{|x|\rightarrow\infty}u(x)
\ee
exists and 
\be
\lim_{|x|\rightarrow\infty}|x||Du(x)|=0.
\ee
Moreover, if not all of the $a_i$ are identical,
\be\label{MultiPoleLimIneq}
\min_{1\le i\le N}a_i<\lim_{|x|\rightarrow\infty}u(x)<\max_{1\le i\le N}a_i.
\ee
\end{cor}
\begin{rem}
Using the estimate from Theorem \ref{DecayThm}, we can also conclude that the limit \eqref{MultiPoleLim}
occurs with at least an algebraic rate of convergence. 
\end{rem}

We can also make a few basic observations about a particular level set of solutions of equation 
\eqref{pMultiPole}.
\begin{cor}
Suppose $u\in {\cal D}^{1,p}(\R^n)$ minimizes \eqref{pDirichlet} subject to the constraints \eqref{MultiConstraint} and
$$
\lim_{|x|\rightarrow\infty}u(x)=\beta.
$$
Then 
$$
\{x\in \R^n: u(x)=\beta\}
$$
is nonempty and noncompact. Furthermore, $c=\beta$ is the only value for which the level set 
$$
\{x\in \R^n: u(x)=c\}
$$
has this property. 
\end{cor}
\begin{proof}
We have established  
$$
\beta\in \left[\min_{1\le i\le N}a_i,\max_{1\le i\le N}a_i\right]=u(\R^n).
$$
Since $u$ is continuous, there is some $z\in \R^n$ for which $u(z)=\beta$.  Consequently, $\{x\in \R^n: u(x)=\beta\}\neq\emptyset$.  

\par If $\{x\in \R^n: u(x)=\beta\}\subset B_R$ for some $R>0$, then either $u> \beta$ in $\R^n\setminus{\overline{B_R}}$ or  $u< \beta$ in $\R^n\setminus{\overline{B_R}}$. If $u> \beta$ in $\R^n\setminus{\overline{B_R}}$, then $u-\beta$ is a bounded and positive $p$-harmonic function 
on an exterior domain. By Proposition  \ref{LimitPosPreserving}, there is a $\eta>0$ such that  $u(x)-\beta\rightarrow \eta$ as $|x|\rightarrow\infty$. However, this contradicts $u(x)\rightarrow \beta$ as $|x|\rightarrow\infty$. Thus, no such $R$ exists and $\{x\in \R^n: u(x)=\beta\}$ 
is noncompact. 

\par Finally, we note that if there is a sequence $(x_k)_{k\in \N}$ with $|x_k|\rightarrow\infty$ and $u(x_k)=c$ then 
$$
\beta=\lim_{k\rightarrow\infty}u(x_k)=c.
$$
That is, $\{x\in \R^n: u(x)=c\}$ is compact when $c\neq \beta$. 
\end{proof}

\begin{rem} It would be really interesting to explicitly compute 
$$
\lim_{|x|\rightarrow\infty}u(x)
$$
for solutions of the multipole PDE \eqref{pMultiPole}. Perhaps it is possible to do so in terms of the given data $a_1,\dots, a_N$ and 
$x_1,\dots, x_N$. Recall that when $N=1$,
$$
\lim_{|x|\rightarrow\infty}u(x)=a_1
$$
by Corollary \ref{LiouvilleProp}; and when $N=2$,
$$
\lim_{|x|\rightarrow\infty}u(x)=\frac{a_1+a_2}{2}
$$
by Theorem \ref{mainThm}. We wonder if there are analogous formulae for $N\ge 3$.
\end{rem}

\begin{figure}[h]
\centering
 \includegraphics[width=1\columnwidth]{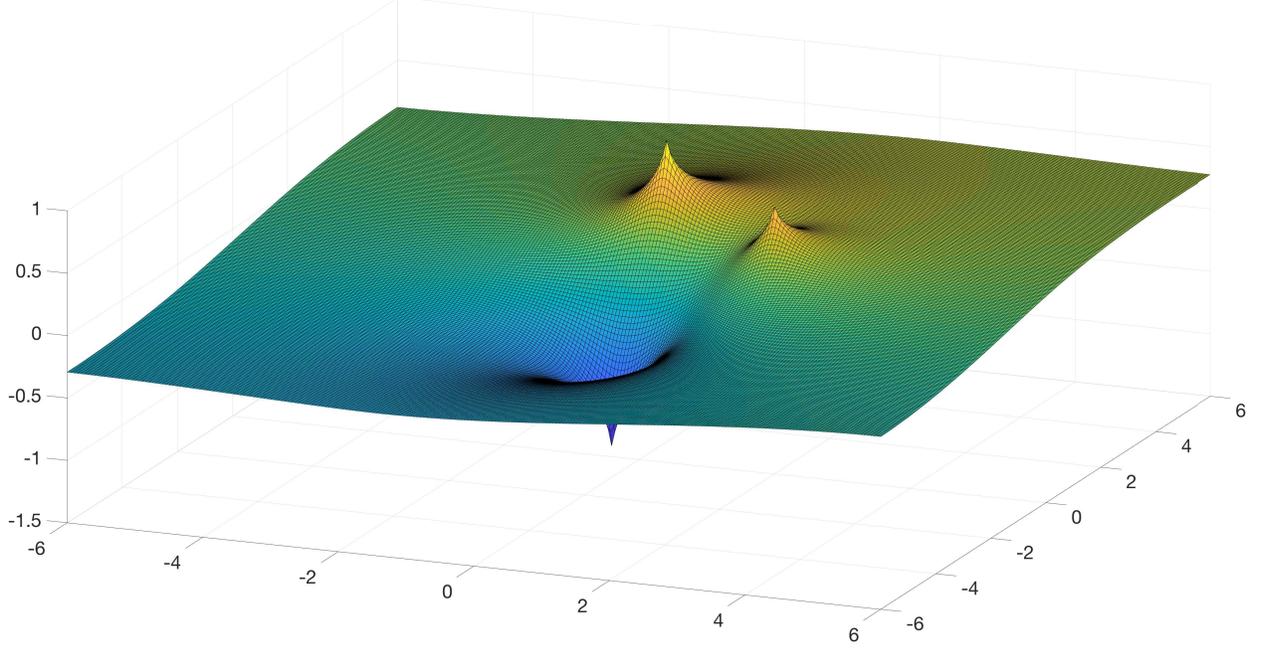}
 \caption{The graph of the solution of the multipole equation \eqref{pMultiPole} with $n=2, p=3, x_1=(0,1), x_2=(0,-1), x_3= (2,0)$ and $c_1=1, c_2=-3/2$ and $c_3=1/2$.}\label{Tripole}
\end{figure}

%\FloatBarrier
\par We conclude by studying the (non)differentiability of minimizers at the points $x_1,\dots, x_N$. This and the other properties we have already discussed about solutions of the multipole PDE may be 
seen in Figures \ref{Tripole} and \ref{Quadpole}.

\begin{figure}[h]
\centering
 \includegraphics[width=1\columnwidth]{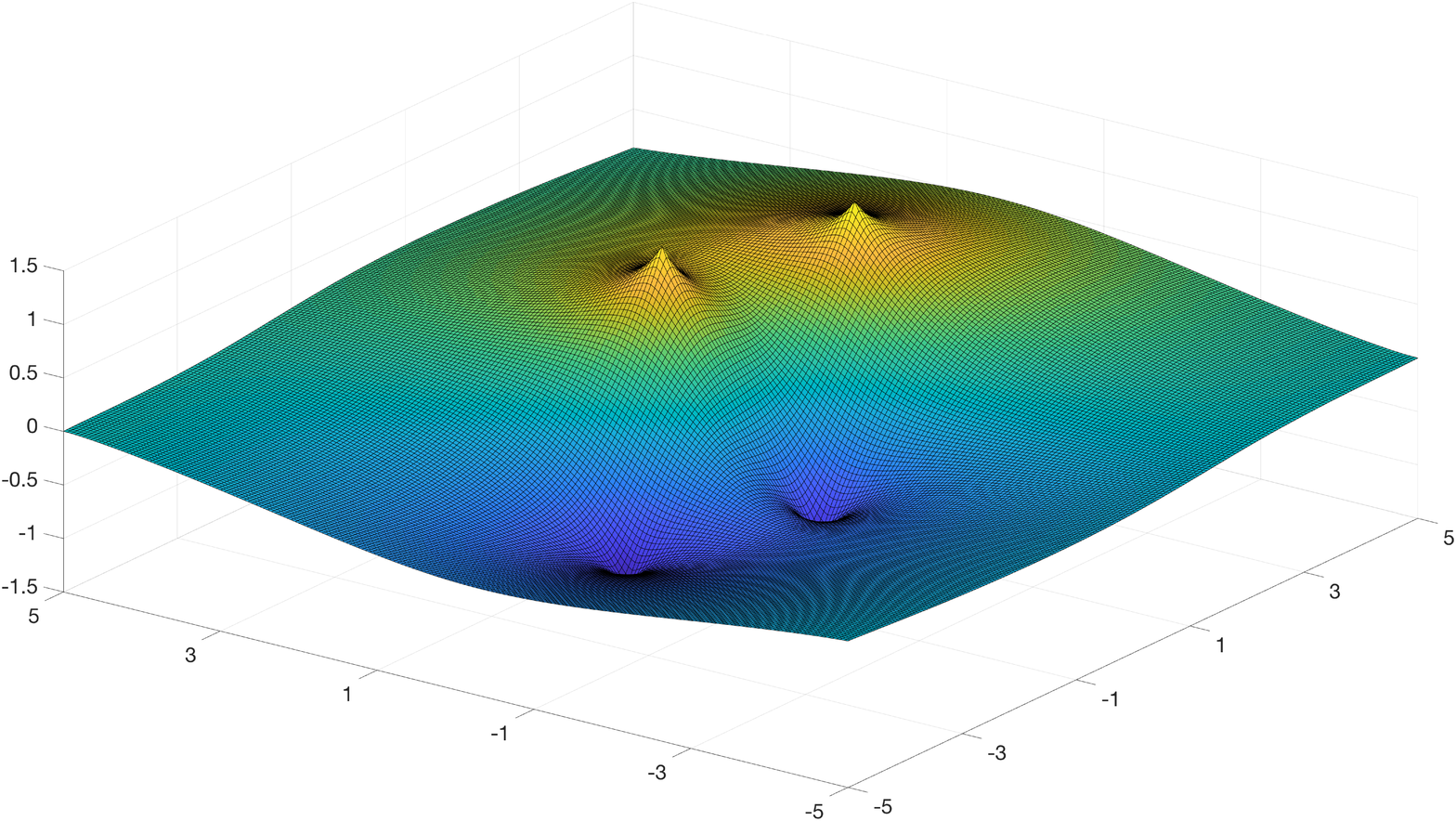}
 \caption{The graph of a solution of the multipole equation \eqref{pMultiPole} with $n=2, p=5, x_1=(0,1), x_2=(0,-1), x_3= (2,0),x_4= (-2,0)$ and $c_1=2, c_2=-2, c_3=1$ and $c_4=-1$ }\label{Quadpole}
\end{figure}

%\FloatBarrier

\begin{prop}\label{NonDiffProp}
Suppose $u\in {\cal D}^{1,p}(\R^n)$ minimizes \eqref{pDirichlet} subject to the constraints \eqref{MultiConstraint} and $i\in \{1,\dots, N\}$. If $u$ has a strict local maximum or minimum at $x_i$, then $u$ is not differentiable at $x_i$.
\end{prop}
\begin{proof}
We will prove that $u$ is not differentiable at $x_1$ provided that it has a strict local max at $x_1$. With this assumption, there is some $r>0$ such that $u(x)<u(x_1)$ for $x\in \overline{B_r(x_1)}\setminus\{x_1\}$. 
In particular,
\be\label{strictLocalMaxIneq}
u(x_1)>\max_{\partial B_r(x_1)}u.
\ee
Choosing $r$ smaller if necessary, we may also suppose that $u$ is $p$-harmonic in $B_r(x_1)\setminus\{x_1\}$. 

\par Set 
$$
v(x):=\left(u(x_1)-\max_{\partial B_r(x_1)}u\right)\left(1-\frac{|x-x_1|^{\frac{p-n}{p-1}}}{r^{\frac{p-n}{p-1}}}\right)+\max_{\partial B_r(x_1)}u
$$
for $x\in \overline{B_r(x_1)}$. Note that $v(x_1)=u(x_1)$ and 
$$
v|_{\partial B_r(x_1)}=\max_{\partial B_r(x_1)}u\ge u|_{\partial B_r(x_1)}.
$$
As $v$ is $p$-harmonic in $B_r(x_1)\setminus\{x_1\}$, comparison gives $v\ge u$ in $B_r(x_1)$. 

\par If $u$ is differentiable at $x_1$,  then
\begin{align}
v(x)&=\left(u(x_1)-\max_{\partial B_r(x_1)}u\right)\left(1-\frac{|x-x_1|^{\frac{p-n}{p-1}}}{r^{\frac{p-n}{p-1}}}\right)+\max_{\partial B_r(x_1)}u\\
&\ge u(x)\\
&=u(x_1)+Du(x_1)\cdot (x-x_1)+o(|x-x_1|)\\
&\ge u(x_1)-\left(|Du(x_1)|+o(1)\right)|x-x_1|
\end{align}
as $x\rightarrow x_1$.  Rearranging this inequality gives
$$
\left(|Du(x_1)|+o(1)\right)|x-x_1|^{1-\left(\frac{p-n}{p-1}\right)}\ge \frac{1}{r^{\frac{p-n}{p-1}}}\left(u(x_1)-\max_{\partial B_r(x_1)}u\right).
$$
And sending $x\rightarrow x_1$ leads us to 
$$
0\ge u(x_1)-\max_{\partial B_r(x_1)}u,
$$
which contradicts \eqref{strictLocalMaxIneq}. Consequently, $u$ is not differentiable at $x_1$. 
\end{proof}

\begin{cor}
Suppose $u\in {\cal D}^{1,p}(\R^n)$ minimizes \eqref{pDirichlet} subject to the constraints \eqref{MultiConstraint} and that $a_1,\dots, a_N$ are not all the same. Then $u$ is not differentiable at any point in which it attains its global maximum or its global minimum. 
\end{cor}
\begin{proof}
Suppose 
$$
a_1=\max_{1\le i\le N}a_i.
$$
We noted that $u(x)<u(x_1)=a_1$ in $\R^n\setminus\{x_1,\dots, x_N\}$ in \eqref{SimplyUpandLowBounds}. It follows that $u$ has a strict local max at $x_1$. 
By Proposition \ref{NonDiffProp}, $u$ isn't differentiable at $x_1$. As a result, $u$ is not differentiable at any point in which it attains its global maximum. We can argue similarly for points at which $u$ attains its global minimum. 
\end{proof}

\appendix

%%%%%%%%%%%%%%%%%%%%%%%%%%%%%%%%%%%%%%%%%%%%%%%%%%%%
\section{Numerical method}
We will now discuss the method used to approximate the solution of PDE \eqref{DipolePDE} displayed in Figure \ref{DipoleFig}. 
It turns out that this method also can be adapted to obtain approximations of solutions of the multipole equation \eqref{pMultiPole}, as exhibited in 
Figures \ref{Tripole} and \ref{Quadpole}.  For simplicity, we will focus on the particular case of 
approximating a solution $u$ of the PDE
\be\label{SpecficDipole}
-\Delta_pu=\delta_{(0,1)}-\delta_{(0,-1)}
\ee
in $\R^2$. We will also change notation and use ordered pairs $(x,y)$ to denote points in $\R^2$ so that $u=u(x,y)$. 

\par Observe that any solution $u\in {\cal D}^{1,p}(\R^2)$ of \eqref{SpecficDipole} minimizes 
\be\label{originalE}
\iint_{\R^2}\frac{1}{p}|Dv|^pdxdy- (v(0,1)-v(0,-1))
\ee
among all $v\in {\cal D}^{1,p}(\R^2)$. For a given $\ell\in \N$, we may also consider the problem of minimizing 
\be\label{TruncatedEnergy}
\int^\ell_{-\ell}\int^\ell_{-\ell}\frac{1}{p}|Dv|^pdxdy- (v(0,1)-v(0,-1))
\ee
amongst $v\in W^{1,p}([-\ell,\ell]^2)$. It is not hard to show this problem has a minimizer $u_\ell\in W^{1,p}([-\ell,\ell]^2)$. Further, 
it is routine to check that $u_\ell(x,y)-u_\ell(0,0)$ converges to $u(x,y)$ for each $(x,y)\in \R^2$ as $\ell\rightarrow\infty$, where $u$ is the unique minimizer of \eqref{originalE} with $u(0,0)=0$.   Consequently, we will focus on approximating $u_\ell$. 

Below we will show how to derive a discrete version of our minimization problem for $u_\ell$. Then we will explain how to use an 
iteration scheme based on a quasi-Newton method.  Again we emphasize that all of the figures in this article 
were based on this method or minor variants to account for differences in the particular examples we considered.

%--------------------------------------------------------------------------------------------
\subsection{Discrete energy}
Let us fix $\ell\in \N$ ($\ell\ge 2$) and discretize the interval $ [-\ell,\ell]$ along the $x$-axis with 
$$
x_i=-\ell+(i-1)h
$$
for $i=1,\dots, M$. Here
$$
h=\frac{2\ell}{M-1},
$$
and we note that each of the subintervals $[x_1,x_{2}],\dots, [x_{M-1},x_{M}] $ has length $h$. We can do the same for the interval $[-\ell,\ell]$ along the $y$-axis and obtain $$y_j=-\ell+(j-1)h$$ for $j=1,\dots, M$.  
Our goal is to derive an appropriate energy specified for functions defined on the grid points $(x_i,y_j)$.

\par To this end, we observe that if $v:[-\ell,\ell]^2\rightarrow \R$ is sufficiently smooth 
\begin{align*}
&\int^\ell_{-\ell}\int^\ell_{-\ell}|Dv|^pdxdy\\
&\hspace{.5in}\approx\sum^{M-1}_{i,j=1}|Dv(x_i,y_j)|^ph^2\\
&\hspace{.5in}=\sum^{M-1}_{i,j=1}\left(v_{x}(x_i,y_j)^2+v_{y}(x_i,y_j)\right)^{p/2}h^2\\
&\hspace{.5in}\approx\sum^{M-1}_{i,j=1}\left( \left(\frac{v(x_i+h,y_j)-v(x_i,y_i)}{h}\right)^2+\left(\frac{v(x_i,y_j+h)-v(x_i,y_i)}{h}\right)^2\right)^{p/2}h^2\\
&\hspace{.5in}=\sum^{M-1}_{i,j=1}\left( \left(\frac{v(x_{i+1},y_j)-v(x_i,y_i)}{h}\right)^2+\left(\frac{v(x_i,y_{j+1})-v(x_i,y_i)}{h}\right)^2\right)^{p/2}h^2\\
&\hspace{.5in}= h^{2-p}\sum^{M-1}_{i,j=1}\left( \left(v(x_{i+1},y_j)-v(x_i,y_i)\right)^2+\left(v(x_i,y_{j+1})-v(x_i,y_i\right)^2\right)^{p/2}.
\end{align*} 
We also suppose $h=1/k$ for some $k\in \N$ which gives
$$
M=2\ell k+1
$$
and
$$
(x_{\ell k+1}, y_{(\ell+1) k+1})=(0,1)\quad \text{and}\quad (x_{\ell k+1}, y_{(\ell-1) k+1})=(0,-1).
$$

As a result, we will attempt to minimize
\be\label{DiscreteE}
E(v)=\frac{1}{p}k^{p-2}\sum^{M-1}_{i,j=1}\left( \left(v_{i+1,j}-v_{i,j}\right)^2+\left(v_{i,j+1}-v_{i,j}\right)^2\right)^{p/2}
-(v_{\ell k+1,(\ell +1)k+1}-v_{\ell k+1,(\ell -1)k+1})
\ee
over the $M^2-1$ variables 
$$
v=
\left(\begin{array}{ccccc}
v_{1,1} & v_{1,2}&\dots& v_{1,M-1}& v_{1,M}\\
v_{2,1} & v_{2,2}&\dots& v_{2,M-1}& v_{2,M}\\
\vdots & \vdots &  \ddots & \vdots & \vdots \\
v_{M-1,1} & v_{M-1,2} & \dots & v_{M-1,M-1} &v_{M-1,M} \\
v_{M,1} & v_{M,2} & \dots & v_{M,M-1} &
\end{array}\right).
$$
A minimizer $v=(v_{ij})$  for $E$ would 
then form an approximation for $u_\ell$ on the grid points $(x_i,y_j)$
$$
u_\ell(x_i,y_j)\approx v_{ij}.
$$

%--------------------------------------------------------------------------------------------
\subsection{Quasi-Newton method}
We used a multidimensional version of the secant method to approximate minimizers of the discrete energy $E$ defined above in \eqref{DiscreteE}. In particular, since 
$E$ is convex we only need to approximate a $v=(v_{ij})$ such that 
$$
\partial_{v_{ij}} E(v)=0
$$
for each $i,j=1,\dots, M$ with $(i,j)\neq (M,M)$.

\par First we chose the initial guesses
$$
v^0_{ij}=0
$$
and 
$$
v^1_{ij}=g(x_i,y_j).
$$
Here
$$
g(x,y)=-\frac{1}{4\pi}\log\left[\frac{x^2+(y-1)^2+10^{-2}}{x^2+(y+1)^2+10^{-2}}\right]
$$
is approximately equal to 
$$
g_0(x,y)=-\frac{1}{4\pi}\log\left[\frac{x^2+(y-1)^2}{x^2+(y+1)^2}\right],
$$
which is a solution of the Dipole equation $-\Delta g_0=\delta_{(0,1)}-\delta_{(0,-1)}$ in $\R^2$. 

\par Then we performed the iteration
\be
\begin{cases}
\displaystyle v^{m+1}_{ij}=v^{m}_{ij}-\tau_m\partial_{v_{ij}} E(v^m)
\\\\
\tau_m:=\frac{\displaystyle\sum_{ij}(v^m_{ij}-v^{m-1}_{ij}) (\partial_{v_{ij}} E(v^m)-\partial_{v_{ij}} E(v^{m-1}))}{\displaystyle
\sum_{ij}\left(\partial_{v_{ij}} E(v^m)-\partial_{v_{ij}} E(v^{m-1})\right)^2}
\end{cases}
\ee
for $m=1,2,3,\dots$ until the stopping criterion 
$$
\max_{ij}\left|\partial_{v_{ij}} E(v^m)\right|<10^{-6}
$$
was achieved.  The iteration was computed for all $i,j=1,\dots, M$ except for $(i,j)\neq (M,M)$.

\bibliography{MorreyAsymBib}{}

\bibliographystyle{plain}

\typeout{get arXiv to do 4 passes: Label(s) may have changed. Rerun}

\end{document}